\documentclass{amsart}
\usepackage{amsfonts}
\usepackage{amssymb}
\usepackage{amscd}

\newtheorem{theorem}{Theorem}[section]
\newtheorem{problem}[theorem]{Problem}
\newtheorem{corollary}[theorem]{Corollary}
\newtheorem{lemma}[theorem]{Lemma}
\newtheorem{proposition}[theorem]{Proposition}
\newtheorem{example}[theorem]{Example}

\theoremstyle{definition}
\theoremstyle{remark}
\theoremstyle{problem}
\numberwithin{equation}{section}

\newcommand{\cl}{\operatorname{cl}}

\newcommand{\intr}{\operatorname{int}}
\newcommand{\adh}{\operatorname{adh}}

\def\0{\varnothing}

\begin{document}
\title{When is the Isbell topology a group topology?}
\author{Szymon Dolecki}
\address{Mathematical Institute of Burgundy\\
Burgundy University, B.P. 47 870, 21078 Dijon, France}
\email{dolecki@u-bourgogne.fr}
\author{Fr\'{e}d\'{e}ric Mynard}
\address{Department of Mathematical Sciences, Georgia Southern University,
PB 8093, Statesboro GA 30460, U.S.A.}
\email{fmynard@georgiasouthern.edu}
\thanks{We are grateful to Professor Ahmed Bouziad (University of Rouen) for
comments that helped us to improve this paper.}
\date{%
\today%
}
\maketitle

\begin{abstract}
Conditions on a topological space $X$ under which the space $C(X,\mathbb{R})$
of continuous real-valued maps with the Isbell topology $\kappa $ is a
topological group (topological vector space) are investigated. It is proved
that the addition is jointly continuous at the zero function in $C_{\kappa
}(X,\mathbb{R})$ if and only if $X$ is infraconsonant. This property is
(formally) weaker than consonance, which implies that the Isbell and the
compact-open topologies coincide. It is shown the translations are
continuous in $C_{\kappa }(X,\mathbb{R})$ if and only if the Isbell topology
coincides with the fine Isbell topology. It is proved that these topologies
coincide if $X$ is prime (that is, with at most one non-isolated point), but
do not even for some sums of two consonant prime spaces.
\end{abstract}

\section{Introduction}

In \cite{Isbell_a} and \cite{Isbell_b} Isbell introduced and studied a
topology on the space $C(X,Z)$ of continuous functions from a topological
space $X$ to a topological space $Z$, defined in terms of (what is now
called) \emph{compact families} of open subsets of $X$ and open subsets of $%
Z $. The \emph{Isbell topology} is finer than the \emph{compact-open}
topology and coarser than the \emph{natural topology} (that is, the
topological reflection of the \emph{natural convergence}, most often called 
\emph{continuous convergence}). Recently Jordan introduced in \cite%
{francis.coincide} several intermediate topologies, finer than the Isbell
and coarser than the natural topology, that turn out to be instrumental in
understanding function spaces. One of them is the so-called \emph{fine
Isbell topology}.

The Isbell topology and the natural topology coincide on $C(X,\$)$ (that can
be identified with the set of closed subsets of $X$) and on the homeomorphic
space $C(X,\$^{\ast })$ (of open subsets of $X$) where it is homeomorphic to
the \emph{Scott topology} \footnote{$\$:=\left\{ \varnothing ,\left\{
1\right\} ,\left\{ 0,1\right\} \right\} $ and $\$^{\ast }:=\left\{
\varnothing ,\left\{ 0\right\} ,\left\{ 0,1\right\} \right\} $ are two
homeomorphic representations of the\emph{\ Sierpi\'{n}ski topology} on $%
\left\{ 0,1\right\} .$}. The open sets for the Scott topology on $C(X,\$)$
are precisely the compact families of open subsets of $X$. A topological
space $X$ is called \emph{consonant} \cite{DGL.kur} if these topologies on $%
C(X,\$^{\ast })$ coincide with the compact-open topology.\footnote{%
In other words, if each compact family on $X$ is compactly generated.}

It is known that if $X$ is consonant, then the Isbell topology on $C(X,%
\mathbb{R})$ coincides with the compact-open topology. We prove that the
converse is true for completely regular spaces, partially answering \cite[%
Problem 62]{openproblems2}. Answering \cite[Problem 61]{openproblems2}
positively, we also show that the Isbell topology on $C(X,Z)$ is completely
regular whenever $Z$ is.

There are consonant examples (e.g., \cite[Example 5.12]{ELS}, \cite%
{GeorgiouIlliadis07}) of spaces, for which the Isbell topology is strictly
coarser than the natural topology, but to our knowledge there is so far no
characterization of $X$ for which the Isbell topology and the natural
topology coincide on $C(X,\mathbb{R})$.

The natural convergence is always a group convergence, in particular, it is
invariant under translations\footnote{%
Actually, the natural convergence is a convergence vector space.}, hence the
natural topology is also invariant under translations as the topological
reflection of the natural convergence (see \cite{DMtransfer}), but need not
be a group topology, e.g. \cite{jarchowbook}. In \cite{Isbellvector}, B.
Papadopoulos proposes a sufficient condition on a topological space $X$ for
the Isbell topological space $C_{\kappa }(X,\mathbb{R)}$ to be a vector
space topology. However, it seems that no example has been known so far of a
space $X$, for which $C_{\kappa }(X,\mathbb{R)}$ is \emph{not }a vector
space topology.

In this note, we investigate under what conditions the Isbell topology is a
group topology, equivalently a vector space topology, because we prove that
multiplication by scalars is jointly continuous for the Isbell topology.

In general, a topology on an abelian group is a group topology if and only
if the translations and the inversion are continuous, and if the group operation is
(jointly) continuous at the neutral group element. As the inversion is a homeomorphism
for the Isbell topology on $C(X,\mathbb{R})$, we are confronted with two
quests about the Isbell topology on $C(X,\mathbb{R})$:

\begin{enumerate}
\item invariance by translations, and

\item continuity of the addition at the zero function $\overline{0}$, that
is, the property 
\begin{equation}
\mathcal{N}_{\kappa }(\overline{0})+\mathcal{N}_{\kappa }(\overline{0})\geq 
\mathcal{N}_{\kappa }(\overline{0}).  \label{eq:isbellat0}
\end{equation}
\end{enumerate}

More specifically, we show that the space $C_{\kappa }(X,\mathbb{R})$ of
real-valued continuous functions on $X$ endowed with the Isbell topology is
invariant under translations if and only if\emph{\ }the Isbell and fine
Isbell topologies coincide. In \cite{francis.coincide}, Jordan provides an
example of a topological space $X$, for which the Isbell and fine Isbell
topologies on $C(X,\mathbb{R})$ do \emph{not }coincide. This shows that
there exists $X$ for which $C_{\kappa }(X,\mathbb{R})$ is not invariant by
translations.

We call a space \emph{infraconsonant }if every compact family $\mathcal{A}$
contains another compact family $\mathcal{B}$ such that every pairwise
intersection of elements of $\mathcal{B}$ belongs to $\mathcal{A}$, and we
show that (\ref{eq:isbellat0}) holds if and only if $X$ is infraconsonant.
Of course, every consonant space is infraconsonant. There are infraconsonant
and non consonant spaces, but we do not know yet of a completely regular one.

\begin{problem}
\label{prob:infranonC} Does there exist a completely regular infraconsonant
space that is not consonant?
\end{problem}

We call a topological space \emph{prime }if it has at most one non-isolated
point. We show that $C_{\kappa }(X,\mathbb{R})$ is invariant under
translations if $X$ is a prime space and that there are prime spaces that
are not infraconsonant. In other words, $C_{\kappa }(X,\mathbb{R})$ may be
translation-invariant without satisfying (\ref{eq:isbellat0}). We also show
that $C_{\kappa }(X,\mathbb{R})$ may fail to have either of these
properties. However, we do not know if it can satisfy (\ref{eq:isbellat0})
without being invariant under translations. In other words:

\begin{problem}
\label{prob:infra/translations} Does there exist a completely regular
infraconsonant space $X$ such that $C_{\kappa }(X,\mathbb{R})$ has
discontinuous translations?
\end{problem}

A positive solution to this problem would also provide a positive answer to
Problem \ref{prob:infranonC}, because $C_{\kappa }(X,\mathbb{R})$ is a
topological group if $X$ is consonant \footnote{%
Problem \ref{prob:infra/translations} has been recently solved in the
negative in \cite{groupisbell}.}. We do not know if the converse is true:

\begin{problem}
\label{prob:groupnoncons} Does there exist a non consonant completely
regular space $X$ such that $C_{\kappa }(X,\mathbb{R})$ is a topological
group?
\end{problem}

In view of our result, a prime positive solution to Problem \ref%
{prob:infranonC} would also provide a positive answer to Problem \ref%
{prob:groupnoncons}.

\section{Generalities}

If $\mathcal{A}$ is a family of subsets of a topological space $X$ then $%
\mathcal{O}_{X}(\mathcal{A})$ denotes the family of open subsets of $X$
containing an element of $\mathcal{A}$. In particular, if $A\subset X$ then $%
\mathcal{O}_{X}(A)$ denotes the family of open subsets of $X$ containing an
element of $A$. A family $\mathcal{A=O}_{X}(\mathcal{A)}$ is \emph{compact }%
if whenever $\mathcal{P}\subset \mathcal{O}_{X}$ and $\bigcup \mathcal{P}\in 
\mathcal{A}$ then there is a finite subfamily $\mathcal{P}_{0}$ of $\mathcal{%
P}$ such that $\bigcup \mathcal{P}_{0}\in \mathcal{A}$. Of course, for each
compact subset $K$ of $X,$ the family $\mathcal{O}_{X}(K)$ is compact. The
following proposition extends the fact that continuous functions are bounded
on compact sets.

\begin{proposition}
\label{prop:bounded} If $\mathcal{A}$ is a compact family on $X$ and $f\in
C(X,\mathbb{R})$, then there is $A\in \mathcal{A}$ such that $f(A)$ is
bounded.
\end{proposition}

\begin{proof}
As $\bigcup\nolimits_{n<\omega }f^{-}(\left\{ r:\left\vert r\right\vert
<n\right\} )=X\in \mathcal{A}$ and $f$ is continuous, there exists $n<\omega 
$ such that $f^{-}(\left\{ r:\left\vert r\right\vert <n\right\} )\in 
\mathcal{A}$ by the compactness of $\mathcal{A}$.
\end{proof}

We denote by $\kappa (X)$ the collection of compact families on $X$. Seen as
a family of subsets of $\mathcal{O}_{X}$ (the set of open subsets of $X$), $%
\kappa (X)$ is the family of open subsets of the \emph{Scott topology};
hence every union of compact families is compact, in particular $%
\bigcup_{K\in \mathcal{K}}\mathcal{O}_{X}(K)$ is compact if $\mathcal{K}$ is
a family of compact subsets of $X$. A topological space is called \emph{%
consonant }if every compact family $\mathcal{A}$ is \emph{compactly generated%
}, that is, there is a family $\mathcal{K}$ of compact sets such that $%
\mathcal{A}=\bigcup_{K\in \mathcal{K}}\mathcal{O}_{X}(K)$. Similarly, $%
k(X):=\{{\mathcal{O}}(K): K\subseteq X \text{ compact }\}$ is a basis for a
topology on ${\mathcal{O}}_X$, and a space $X$ is consonant if and only if
this topology coincides with the Scott topology. Bouziad calls \emph{weakly
consonant }\cite{bouziad.borel}\emph{\ }a space $X$ in which for every
compact family $\mathcal{A}$ there is a compact subset $K$ of $X$ such that $%
\mathcal{O}_{X}(K)\subseteq \mathcal{A}$.

\begin{lemma}
\label{lem:nowhereconsonant}\cite[Lemma 36]{francis.coincide} Consonance and
weak consonance are equivalent among regular topological spaces.
\end{lemma}

The \emph{Isbell topology }on $C(X,Z)$ can be defined by the following
subbase of open sets%
\begin{equation*}
\lbrack \mathcal{A},U]:=\{f\in C(X,Z):\exists A\in \mathcal{A},f(A)\subseteq
U\},
\end{equation*}%
where $\mathcal{A}$ ranges over $\kappa (X)$ and $U$ ranges over open
subsets of $Z$. We write $C_{\kappa }(X,Z)$ for the set $C(X,Z)$ endowed
with the Isbell topology, and $C_{k}(X,Z)$ if it is endowed with the
compact-open topology. A space $X$ is called $Z$-\emph{consonant} if $%
C_{\kappa }(X,Z)=C_{k}(X,Z)$. Note that if $\$$ denotes the Sierpi\'{n}ski
space, then $\$$-consonant means consonant. Moreover, the following is
immediate.

\begin{proposition}
\label{pro:consonant} $X$ is consonant if and only if it is $Z$-consonant
for every $Z$. In particular, if $X$ is consonant, then $C_{\kappa }(X,%
\mathbb{R})$ is a topological vector space.
\end{proposition}

\cite[Problem 62]{openproblems2} asks for what spaces $Z$ (other than $\$$)
does $Z$-consonance imply consonance. We have the following partial answer,
which refines \cite[Theorem 4.4]{DGL.CRAS} which was announced without proof
and proved in \cite[Theorem 4.17]{mynard.uk} \footnote{%
Note that the notion of $\mathbb{R}$-consonance introduced in \cite%
{mynard.uk} (coincidence of the natural and compact-open topologies on $C(X,%
\mathbb{R})$) is stronger than our notion and should not be confused.}.

The \emph{grill of }a family $\mathcal{A}$ of subsets of $X$ is the family $%
\mathcal{A}^{\#}:=\{B\subseteq X:\forall A\in \mathcal{A}$, $A\cap B\neq
\varnothing \}$. Note that if $\mathcal{A}=\mathcal{O}(\mathcal{A)}$, then 
\begin{equation*}
A\in \mathcal{A\Longleftrightarrow }A^{c}\notin \mathcal{A}^{\#}\text{.}
\end{equation*}

\begin{proposition}
\label{pro:Rconsonant} If $X$ is completely regular and $\mathbb{R}$%
-consonant, then it is consonant.
\end{proposition}

\begin{proof}
If $X$ is $\mathbb{R}$-consonant then in particular $\mathcal{N}_{k}(%
\overline{0})\geq \mathcal{N}_{\kappa }(\overline{0})$ where $\overline{0}$
denotes the zero function. Hence for every $\mathcal{A}\in \kappa (X)$ there
exist a compact subset $K$ of $X$ and $r>0$ such that $[K,B_{r}]\subseteq %
\left[ \mathcal{A},B_{\frac{1}{2}}\right] $ where $B_{r}:=\left( -r,r\right) 
$. In view of Lemma \ref{lem:nowhereconsonant}, it is sufficient to show
that $\mathcal{O}(K)\subseteq \mathcal{A}$. Assume on the contrary that
there is an open set $U$ such that $K\subseteq U$ and $U\notin \mathcal{A}$.
Then the closed set $F:=X\setminus U$ is disjoint from $K$ and $F\in 
\mathcal{A}^{\#}$. As $X$ is completely regular, there is $h\in C(X,\mathbb{R%
})$ such that $h(K)=\{0\}$ and $h(F)=\{1\}$. Then $h\in \lbrack K,B_{r}]$
but $h\notin \left[ \mathcal{A},B_{\frac{1}{2}}\right] $ because $1\in h(A)$
for every $A\in \mathcal{A}$; a contradiction.
\end{proof}

In the proof above, we used the well-known fact that \emph{if }$A$\emph{\ is
a compact subset of a completely regular space }$X$\emph{\ and }$F$\emph{\
is a closed subset of }$X$\emph{\ such that }$A\cap F=\varnothing $\emph{,
then there exists }$h\in C(X,[0,1])$\emph{\ such that }$h(A)=\{0\}$\emph{\
and }$h(F)=\{1\}$. We extend this fact to a closed set and a compact family.

\begin{lemma}
\label{funct:comp_sep} If $\mathcal{A}=\mathcal{O}(\mathcal{A})$ is a
compact family of subsets of a completely regular topological space $X$, and 
$F$ is a closed subset of $X$ with $F^{c}\in \mathcal{A}$, then there is $%
A\in \mathcal{A}$ and $h\in C(X,[0,1])$ such that $h(A)=\{0\}$ and $%
h(F)=\{1\}$.
\end{lemma}

\begin{proof}
By complete regularity, for every $x\notin F$, there is an open neighborhood 
$O_{x}$ of $x$ and $h_{x}\in C(X,[0,1])$ such that $h_{x}(O_{x})=\{0\}$ and $%
h_{x}(F)=\{1\}$. Therefore $F^{c}=\bigcup_{x\notin F}O_{x}\in \mathcal{A}$,
so that by the compactness of $\mathcal{A}$ there is $n<\omega $ and $%
x_{1},\ldots ,x_{n}\notin F$ such that $A=\bigcup_{1\leq i\leq
n}O_{x_{i}}\in \mathcal{A}$. The continuous function $\min_{1\leq i\leq
n}h_{x_{i}}$ is $0$ on $A$ and $1$ on $F$.
\end{proof}

Consequently, for each $O\in \mathcal{A}$ there exists a continuous function 
$h$ valued in $\left[ 0,1\right] $ such that $O\supset \left\{ x:h(x)<\frac{1%
}{2}\right\} \supset \left\{ x:h(x)=0\right\} \supset \intr\left\{
x:h(x)=0\right\} \in \mathcal{A}$, that is,

\begin{corollary}
An (openly isotone) compact family of subsets of a completely regular
topological space has a base of co-zero sets and a base of the interiors of
zero sets.
\end{corollary}

Papadopoulos says that a space $X$ has \emph{property }$(A^{\ast })$ if
whenever $\mathcal{A\in \kappa (}X)$ and $A_{1}$ and $A_{2}$ are open
subsets of $X$ such that $A_{1}\cup A_{2}\in \mathcal{A}$, there exists
filters $\mathcal{F}_{i}$ such that $A_{i}\in \mathcal{F}_{i}$, $i=1,2$ such
that $\mathcal{O}_{X}(\mathcal{F}_{i})\in \kappa (X)$ and $\mathcal{O}(%
\mathcal{F}_{1})\cap \mathcal{O}(\mathcal{F}_{2})\subseteq \mathcal{A}$. The
main result of \cite{Isbellvector} is that property $(A^{\ast })$ is
sufficient for the Isbell topology on $C(X,\mathbb{R})$ to be a vector space
topology. If $X$ is regular, this result follows immediately from
Proposition \ref{pro:consonant} because of the following:

\begin{proposition}
Let $X$ be a regular topological space. Then $X$ is consonant if and only if 
$X$ has property $(A^{\ast }).$
\end{proposition}

\begin{proof}
Assume that $X$ is consonant and that $A_{1}\cup A_{2}\in \mathcal{A}$ where 
$\mathcal{A\in \kappa (}X).$ Because $X$ is consonant, there is a compact
set $K\subseteq A_{1}\cup A_{2}$ such that $\mathcal{O}(K)\subseteq \mathcal{%
A}$. By regularity and compactness, there are finitely many closed sets $C_i$
such that each $C_i$ is a subset of either $A_1$ or $A_2$ and $K\subseteq
\cup_{i=1}^{i=n}C_i$. Therefore, there exist compact subsets $K_{1}$ of $%
A_{1}$ and $K_{2}$ of $A_{2}$ such that $K=K_{1}\cup K_{2}$, so that $%
\mathcal{O}(K_{1})$ and $\mathcal{O}(K_{2})$ are the sought compact filters.
Conversely, if $X$ satisfies $(A^{\ast })$ then for every $A\in \mathcal{A}$
there is a compact filter $\mathcal{F}$ such that $A\in \mathcal{F}$ and $%
\mathcal{O}(\mathcal{F)\subseteq A}$. Because $\mathcal{F}$ is a compact
filter in a regular space, $\mathcal{O}(\mathcal{F)}=\mathcal{O}\left( \adh\mathcal{F}\right) $ and $\adh\mathcal{F}$ is compact (e.g. \cite[%
Proposition 2.2]{DGL.kur}). Therefore, $\mathcal{A}$ is compactly generated
and $X$ is consonant.
\end{proof}

\begin{lemma}
\label{lem:meshclosed}\cite{dolecki.pannonica} If $\mathcal{A\in \kappa }(X)$
and $C$ is a closed subset of $X$ such that $C\in \mathcal{A}^{\#}$ then%
\begin{equation*}
\mathcal{A}\vee C:=\mathcal{O}\left( \{A\cap C:A\in \mathcal{A\}}\right)
\end{equation*}%
is a compact family on $X.$
\end{lemma}

\begin{lemma}
\label{lem:restriction}If $\mathcal{A\in \kappa }(X)$ and $A_{0}\in \mathcal{%
A}$ then%
\begin{equation*}
\mathcal{A}\downarrow A_{0}:=\mathcal{O}\left( \{A\in \mathcal{A}:A\subseteq
A_{0}\}\right)
\end{equation*}%
is a compact family on $X$.
\end{lemma}

\begin{proof}
If $\bigcup_{i\in I}O_{i}\in \mathcal{A}\downarrow A_{0}$ then there is $%
A\in \mathcal{A}$ such that $A\subseteq A_{0}$ and $A\subseteq \bigcup_{i\in
I}O_{i}$ so that $A\subseteq \bigcup_{i\in I}\left( O_{i}\cap A_{0}\right) $%
. By compactness of $\mathcal{A}$ there is a finite subset $F$ of $I$ such
that $\bigcup_{i\in F}\left( O_{i}\cap A_{0}\right) \in \mathcal{A}$. But $%
\bigcup_{i\in F}\left( O_{i}\cap A_{0}\right) \subseteq A_{0}$ so that $%
\bigcup_{i\in F}\left( O_{i}\cap A_{0}\right) \in \mathcal{A}\downarrow
A_{0} $ and $\bigcup_{i\in F}O_{i}\in \mathcal{A}\downarrow A_{0}$.
\end{proof}

The following theorem answers \cite[Problem 61]{openproblems2}.

\begin{theorem}
\label{thm:dual_Tikh} If $Z$ is completely regular, then $C_{\kappa }(X,Z)$
is completely regular.
\end{theorem}

\begin{proof}
Let $f\in \lbrack \mathcal{A},O]$ where $\mathcal{A}\in \kappa (X)$ and $O$
is $Z$-open. As $\mathcal{A}$ is compact and $f$ continuous, $\mathcal{O}%
_{Z}(f(\mathcal{A}))$ is compact, and since $Z$ is completely regular, by
Lemma \ref{funct:comp_sep}, there is $A\in \mathcal{A}$ and $h\in C(Z,[0,1])$
such that $h(f(A))=\{0\}$ and $h(Z\backslash O)=\{1\}$. Define%
\begin{equation*}
F(g):=\inf_{A\in \mathcal{A}}\sup_{x\in A}h(g(x))=\sup_{H\in \mathcal{A}%
^{\#}}\inf_{x\in H}h(g(x))
\end{equation*}%
for each $g\in C(X,Z)$. Then $F(f)=0$ and $F(g)=1$ for each $g\notin \lbrack 
\mathcal{A},O]$. Moreover, $F:C_{\kappa }(X,Z)\rightarrow \lbrack 0,1]$ is
continuous. To see that $F^{-}\left( [0,r)\right) $ is open for each $r\in
\lbrack 0,1]$, notice that $F(g)<r$ if and only if there is $A_{r}\in 
\mathcal{A}$ such that $g(A_{r})\subset \lbrack 0,r),$ that is, if and only
if $g\in \left[ \mathcal{A},h^{-}\left( [0,r)\right) \right] $. On the other
hand, if $0\leq s<1$ and $s<F(g)$, then, by the second equality, there exist 
$s<t<F(g)$ and a closed set $H\in \mathcal{A}^{\#}$ such that $t\leq h(g(x))$
for each $x\in H$, thus $g(H)\subset h^{-}(s,1]$. By Lemma \ref%
{lem:meshclosed}, $\mathcal{A}\vee H$ is compact, and if an open set
includes $H$ then it belongs to $\mathcal{A}\vee H$, in particular $%
g^{-}h^{-}(s,1]\in \mathcal{A}\vee H$, that is, $g$ belongs to the open set $%
\left[ \mathcal{A}\vee H,h^{-}(s,1]\right] .$ If now $b\in \left[ \mathcal{A}%
\vee H,h^{-}(s,1]\right] $, then there is $A\in \mathcal{A}$ such that $%
h(b(A\cap K))\subset (s,1]$, hence 
\begin{equation*}
s<\sup\limits_{A\in \mathcal{A}}\inf_{x\in A\cap H}h(b(x))\leq \inf_{A\in 
\mathcal{A}}\sup_{x\in A\cap H}h(b(x))\leq \inf_{A\in \mathcal{A}}\sup_{x\in
A}h(b(x))=F(b).
\end{equation*}
\end{proof}

As $[\mathcal{A},-U]=-[\mathcal{A},U]$ for every $U\subset \mathbb{R}$ and
each compact family $\mathcal{A}$, the inversion is a homeomorphism in $%
C_{\kappa }(X,\mathbb{R})$. More generally,

\begin{proposition}
\label{prop:scalarmultIsbell}Multiplication by scalars is jointly continuous
for the Isbell topology.
\end{proposition}

\begin{proof}
Let $f\in C(X,\mathbb{R})$ and $r\in \mathbb{R}$ be such that $rf\in \lbrack 
\mathcal{A},O]$, where $\mathcal{A}$ is a compact family on $X$ and $O$ is
an open subset of $\mathbb{R}$. If $r=0$ then it is enough to consider $%
O=B(0,\varepsilon )$ with $\varepsilon >0$. By Proposition \ref{prop:bounded}
there exist $A_{0}\in \mathcal{A}$ and $R>0$ such that $f(A_{0})\subset
B(0,R)$ and thus $f(A)\subset B(0,R)$ for each basic element $A$ of $%
\mathcal{A}_{0}:=\mathcal{A}\downarrow A_{0}$. Therefore $f\in \lbrack 
\mathcal{A}_{0},B(0,R)]$ and $B(0,\frac{\varepsilon }{R})[\mathcal{A}%
_{0},B(0,R)]\subset \lbrack \mathcal{A},O]$. Let $\left\vert r\right\vert >0$%
. Since $\mathcal{O}\left( \left( rf\right) (\mathcal{A})\right) $ is a
compact family of the consonant space $\mathbb{R},$ there exists a compact
subset $K$ of $O$ such that $\mathcal{O}_{\mathbb{R}}(K)\subset \mathcal{O}%
\left( \left( rf\right) (\mathcal{A})\right) $, hence there exist $A_{0}\in 
\mathcal{A}$ and $\varepsilon >0$ such that $\left( rf\right) (A_{0})\subset
B(K,\varepsilon )\subset B(K,2\varepsilon )\subset O$. If $\mathcal{A}_{0}:=%
\mathcal{A}\downarrow A_{0}$, then $f(A)\subset \frac{1}{r}B(K,\varepsilon
)\ $for a base of elements $A$ of $\mathcal{A}_{0}$, hence $f\in \left[ 
\mathcal{A}_{0},\frac{1}{r}B(K,\varepsilon )\right] $. On the other hand,
there is $\delta >0$ such that $B\left( 1,\frac{\delta }{\left\vert
r\right\vert }\right) B(K,\varepsilon )\subset B(K,2\varepsilon )$ and thus $%
B(r,\delta )\left[ \mathcal{A}_{0},\frac{1}{r}B(K,\varepsilon )\right]
\subset O$.
\end{proof}

\begin{corollary}
If $C_{\kappa }(X,\mathbb{R})$ is a topological group then it is a
topological vector space.
\end{corollary}

\section{Structure of $C_{\protect\kappa }(X,\mathbb{R})$ at the zero
function}

As usual, if $A$ and $B$ are subsets of a group, $A+B:=\{a+b:a\in A,b\in B\}$
and if $\mathcal{A}$ and $\mathcal{B}$ are two families of subsets, $%
\mathcal{A}+\mathcal{B}:=\{A+B:A\in \mathcal{A},B\in \mathcal{B}\}$.

As we have mentioned, a topology on an abelian group is a group topology if
and only if translations are continuous and $\mathcal{N}(0)+\mathcal{N}%
(0)\geq \mathcal{N}(0).$ In this subsection, we investigate the latter
property, that is, 
\begin{equation}
\mathcal{N}_{\kappa }(\overline{0})+\mathcal{N}_{\kappa }(\overline{0})\geq 
\mathcal{N}_{\kappa }(\overline{0}),  \label{eq:at0}
\end{equation}%
for the space $C_{\kappa }(X,\mathbb{R})$. If $(p_{n})$ is a decreasing
sequence of positive numbers that tends to zero, then%
\begin{equation*}
\left[ \bigcap_{i=1}^{n}\mathcal{A}_{i},\left(
-\max_{i=1}^{n}p_{i},\max_{i=1}^{n}p_{i}\right) \right] \subseteq
\bigcap_{i=1}^{n}\left[ \mathcal{A}_{i},\left( -p_{i},p_{i}\right) \right] ,
\end{equation*}%
and thus $\mathcal{N}_{\kappa }(\overline{0})$ has a filter base of the form%
\begin{equation}
\left\{ \left[ \mathcal{A},\left( -p_{n},p_{n}\right) \right] :\mathcal{A}%
\in \kappa (X),n\in \mathbb{N}\right\} ,  \label{eq:kat0}
\end{equation}%
because a finite intersection of compact families is compact.

We call a topological space $X$ \emph{infraconsonant }if for every compact
family $\mathcal{A}$ on $X$ there is a compact family $\mathcal{B}$ such
that $\mathcal{B\vee B}:=\{B\cap C:B\in \mathcal{B},C\in \mathcal{B}\}$ is a
(not necessarily compact) subfamily of $\mathcal{A}$. Note that if $X$ is
consonant then every compact family includes a compact filter of the form $%
\mathcal{O}(K)$ for a compact set $K$. Taking $\mathcal{B=O}(K)$ gives
infraconsonance, so that every consonant space is infraconsonant.

\begin{theorem}
\label{th:at0} Let $\left( G,+\right) $ be an abelian topological group. If $%
X$ is infraconsonant, then the addition is continuous at $\overline{0}$ in $%
C_{\kappa }(X,G)$. Moreover if $X$ is completely regular, then the addition
is continuous at $\overline{0}$ in $C_{\kappa }(X,\mathbb{R})$ if and only
if $X$ is infraconsonant.
\end{theorem}

\begin{proof}
Assume that $X$ is infraconsonant. Let $\mathcal{A\in \kappa }(X)$ and $V\in 
\mathcal{N}_{G}(0).$ By infraconsonance, there exist$\mathcal{\ }$a compact
subfamily $\mathcal{B}$ of $\mathcal{A}$ such that $\mathcal{B}\vee \mathcal{%
B}\subseteq \mathcal{A}$. If $W\in \mathcal{N}_{G}(0)$ such that $%
W+W\subseteq V$, then $[\mathcal{B},W]+[\mathcal{B},W]\subseteq \lbrack 
\mathcal{A},V]$, which proves (\ref{eq:at0})$.$

Conversely, assume that $X$ is not infraconsonant. Let $\mathcal{A}$ be a
compact family witnessing the definition of non infraconsonance. Note that $%
\mathcal{B\vee C\nsubseteq A}$ for every pair of compact families $\mathcal{B%
}$ and $\mathcal{C}$ for otherwise $\mathcal{D}=\mathcal{B\cap C}$ would be
a compact subfamily of $\mathcal{A}$ such that $\mathcal{D\vee D\subseteq A}$%
. Let $V=\left( -\frac{1}{2},\frac{1}{2}\right) .$ We claim that for any
pair $(\mathcal{B},\mathcal{C)}$ of compact families and any pair $(U,W)$ of 
$\mathbb{R}$-neighborhood of $0,$ $[\mathcal{B},U]+[\mathcal{C},W]\nsubseteq
\lbrack \mathcal{A},V]$. Indeed, there exist $B\in \mathcal{B}$ and $C\in 
\mathcal{C}$ such that $B\cap C\notin \mathcal{A}$. Then $B^{c}\cup C^{c}\in 
\mathcal{A}^{\#}$. Moreover, $B^{c}\notin \mathcal{B}^{\#}$ so that by Lemma %
\ref{funct:comp_sep}, there exist $B_{1}\in \mathcal{B}$ and $f\in C(X,%
\mathbb{R})$ such that $f(B_{1})=\{0\}$ and $f(B^{c})=\{1\}.$ Similarly, $%
C^{c}\notin \mathcal{C}$ so that there exist $C_{1}\in \mathcal{C}$ and $%
g\in C(X,\mathbb{R})$ such that $g(C_{1})=\{0\}$ and $g(C^{c})=\{1\}.$ Then $%
f+g\in \lbrack \mathcal{B},U]+[\mathcal{C},W]$ but $1\in (f+g)(A)$ for all $%
A\in \mathcal{A}$ so that $f+g\notin \lbrack \mathcal{A},V]$.
\end{proof}

Complete regularity cannot be relaxed (to regularity) in Theorem \ref{th:at0}%
.

\begin{example}
There exist regular non-infraconsonant spaces $X$, for which the addition is
jointly continuous at $\overline{0}$ in $C_{\kappa }(X,\mathbb{R})$. In \cite%
{HH} Herrlich builds a regular space, on which each continuous function is
constant. To this purpose, for each regular space $Y$ he constructs a
regular space $H(Y)$ such that $Y$ is closed in $H(Y)$, and each $f\in
C\left( H(Y),\mathbb{R}\right) $ is constant on $Y$. Define $%
H^{0}(Y):=Y,H^{n+1}(Y):=H\left( H^{n}\left( Y\right) \right) $ and $%
X:=\bigcup\nolimits_{n<\omega }H^{n}(Y)$ with the finest topology for which
all the injections are continuous. Then each continuous
(real-valued) function on $X$ is constant. Moreover it can be shown that $X$ is regular. This fact is stated in \cite{HH} in case where $Y$ is a singleton, but is true
for an arbitrary regular space $Y$. Let $Y$ be a regular non-infraconsonant
space, for instance the space from Example \ref{ex:Arens}. As all continuous
functions on $X$ are constant, the continuity of the (joint) addition on $%
C_{\kappa }(X,\mathbb{R})$ follows from the continuity of the addition on $%
\mathbb{R}$. If $X$ were infraconsonant, then its closed subset $Y$ would be
infraconsonant, in contradiction with the assumption.
\end{example}

As we have mentioned, $C_{\kappa }(X,\$^{\ast })$ is the lattice of open
subsets of $X$ endowed with the Scott topology, in which open sets are
exactly the compact families of open subsets of $X$. Dually, $C_{\kappa
}(X,\$)$ is the set of closed subsets of $X$ endowed with the \emph{upper
Kuratowski topology, }in which $\mathcal{F}$ is open if the family $\mathcal{%
F}_{c}=\{X\setminus F:F\in \mathcal{F\}}$ is compact. The following was
prompted by a conversation with Ahmed Bouziad (University of Rouen) in June
2008 (in Erice), who asked us if infraconsonance was related to the joint
continuity of the union operation on $C_{\kappa }(X,\$)$.

\begin{lemma}
\label{lem:everywhere} If $X$ is regular and infraconsonant, then for every $%
\mathcal{A}\in \kappa (X)$ and every $A\in \mathcal{A},$ there is $\mathcal{C%
}\in \kappa (X)$ such that $A\in \mathcal{C}$ and $\mathcal{C\vee C\subseteq
A}$.
\end{lemma}

\begin{proof}
Assume that $X$ is infraconsonant and regular. If $\mathcal{A}$ is a compact
family on $X$, then for each element $A$ of $\mathcal{A}$ there is $A_{0}\in 
\mathcal{A}$ such that $\cl A_{0}\subset A$. The family $\mathcal{A}%
_{1}=\mathcal{A}\downarrow A_{0}$ is compact by Lemma \ref{lem:restriction}
so that there is a compact family $\mathcal{B}$ such that $\mathcal{B\vee
B\subseteq A}_{1}$. For each $B_{1}$ and $B_{2}$ in $\mathcal{B}$, there is $%
B\in \mathcal{A},$ $B\subseteq A_{0}$ such that $B\subseteq B_{1}\cap B_{2}.$
Therefore the family $\mathcal{B}_{1}=\mathcal{B}\vee \cl A_{0}$
contains $A$, satisfies $\mathcal{B}_{1}\mathcal{\vee B}_{1}\mathcal{%
\subseteq A}_{1}\mathcal{\subseteq A}$, and is compact by Lemma \ref%
{lem:meshclosed}.
\end{proof}

We shall consider binary maps: the intersection $\cap $ and the union $\cup $%
, defined by $\cap (A,B):=A\cap B$ and $\cup (A,B):=A\cup B$.

\begin{proposition}
\label{prop:join}Let $X$ be a regular topological space. The following are
equivalent:

\begin{enumerate}
\item $X$ is infraconsonant;

\item The intersection $\cap :C_{\kappa }(X,\$^*)\times C_{\kappa
}(X,\$^*)\rightarrow C_{\kappa }(X,\$^*)$ is (jointly) continuous for the
Scott topology;

\item The union $\cup :C_{\kappa }(X,\$)\times C_{\kappa }(X,\$)\rightarrow
C_{\kappa }(X,\$)$ is (jointly) continuous for the upper Kuratowski topology.
\end{enumerate}
\end{proposition}

\begin{proof}
The equivalence between (2) and (3) is immediate, because these topologies
are isomorphic by complementation. Assume $X$ is infraconsonant and let $U$
and $V$ be two open subsets of $X.$ Let $\mathcal{A}$ be a Scott open
neighborhood of $\cap (U,V)$, i.e., a compact family containing $U\cap V$.
By Lemma \ref{lem:everywhere}, there is a compact family $\mathcal{C}$
containing $U\cap V$ such that 
\begin{equation*}
\cap (\mathcal{C},\mathcal{C})=\mathcal{C\vee C\subseteq A}.
\end{equation*}%
Note that $\mathcal{C}$ is a common Scott neighborhood of $U$ and $V$ so
that $\cap $ is continuous.

Conversely, assume that $\cap $ is continuous and consider a compact family $%
\mathcal{A}.$ Since $\cap ^{-1}(\mathcal{A})$ has non-empty interior there
are compact families $\mathcal{B}$ and $\mathcal{C}$ such that $\mathcal{%
B\times C}\subseteq \cap ^{-1}(\mathcal{A})$. The compact family $\mathcal{%
D=B\cap C}$ then satisfies $\mathcal{D\vee D\subseteq A}$ so that $X$ is
infraconsonant.
\end{proof}

Note that the implication $(2)\Longrightarrow (1)$ does not use regularity.
It is well known (e.g. \cite{schwarz84}) that the natural convergence on $%
C(X,{\$})$ is topological if and only if $X$ is core-compact. A topological
space $X$ is \emph{core-compact} if for every $x\in X$ and every $U\in 
\mathcal{O}(x)$ there is $V\in \mathcal{O}(x)$ that is relatively compact in 
$U$.

Therefore, if $X$ is core-compact then the Isbell topology and the natural
convergence coincide on $C(X,\$)$, which is easily seen to make the map $%
\cap $ jointly continuous. In view of Proposition \ref{prop:join}, $X$ is
then infraconsonant. Moreover $X$ is locally compact if and only if the
natural convergence coincides with the compact-open topology on $C(X,\$)$
(e.g. \cite[Proposition 2.19]{schwarz84}), that is, every compact family is
compactly generated. Therefore, if $X$ is core-compact but not locally
compact, then $X $ is infraconsonant but not consonant. Such a space is
constructed in \cite[Section 7]{hofflaw}.

\begin{corollary}
There exists a (non Hausdorff) infraconsonant space that is not consonant.
\end{corollary}

We will now exhibit a class of prime (hence completely regular)
non-infraconsonant spaces. Recall that if $(\mathcal{F}_{n})_{n<\omega }$ is
a sequence of filters, then the \emph{contour }$\mathcal{F}$ is defined by $%
\mathcal{F}:=\bigcup_{p<\omega }\bigcap_{n\geq p}\mathcal{F}_{n}$ \cite%
{cascades}. A prime space $X$ (with only non-isolated point $\infty $) is a 
\emph{contour space }if there exists a family $\{X_{n}:n\in \omega \}$ of
disjoint infinite subsets of $X\setminus \{\infty \}$ such that $%
X=\bigcup_{n<\omega }X_{n}\cup \left\{ \infty \right\} $ and 
\begin{equation*}
\mathcal{N}(\infty )=\{\infty \}\wedge \bigcup_{p<\omega }\bigcap_{n\geq p}%
\mathcal{F}_{n},
\end{equation*}%
where each $\mathcal{F}_{n}$ is a free filter on $X_{n}$. Notice that the
sets $\{\infty \}\cup \bigcup_{n\geq p}F_{n}$, where $F_{n}\in \mathcal{F}%
_{n}$ and $p<\omega $ form a filter base of $\mathcal{N}(\infty )$. Therefore%
\begin{gather}
\forall _{n<\omega }\;X_{n}\notin \mathcal{N}(\infty )^{\#},  \label{1} \\
\forall _{V\in \mathcal{N}(\infty )}\;|\{n\in \omega :X_{n}\cap
V=\varnothing \}|<\omega .  \label{2}
\end{gather}

Compact sets in a contour space are finite. In fact, if $K$ is compact, then 
$K_{n}:=K\cap X_{n}$ is finite, because $\mathcal{F}_{n}$ is finer than the
cofinite filter of $X_{n}$ for each $n<\omega $; as $\mathcal{N}(\infty
)\geq \bigcap_{n\in \omega }\mathcal{F}_{n}$ the set $\bigcup_{n<\omega
}K_{n}$ is closed and does not contain $\infty $, so that it consists of
isolated points, and thus is finite.

In particular, $\infty $ is a \emph{compact-repellent} point, that is, $%
\infty \notin \mathrm{cl}\left( K\setminus \left\{ \infty \right\} \right) $
for each compact set. On the other hand, $X_{n}$ is closed for each $n$ and
the upper and the lower limit of $\left( X_{n}\right) _{n<\omega }$ coincide
and are equal to $\left\{ \infty \right\} $;\footnote{%
that is, $\left\{ \infty \right\} =\bigcap_{p<\omega }\mathrm{cl}\left(
\bigcup_{n\geq p}X_{n}\right) =\bigcap_{N\in \lbrack \omega ]^{\omega }}%
\mathrm{cl}\left( \bigcup_{n\in N}X_{n}\right) $} Thus, by \cite[Theorem 6.1]%
{DGL.kur}, contour spaces are not consonant. Actually,

\begin{theorem}
\label{th:comtour} Contour prime spaces are not infraconsonant.
\end{theorem}

\begin{proof}
The family%
\begin{equation*}
\mathcal{D}:=\{D\mathcal{\in O}(\infty ),\forall _{n\in \omega }:D\cap
X_{n}\neq \varnothing \}
\end{equation*}%
is non-empty and compact. Indeed, if $\{O_{\alpha }:\alpha \in I\}$ is an
open cover of $D\in \mathcal{D},$ there is $\alpha _{0}\in I$ such that $%
\infty \in O_{\alpha _{0}}$ and by (\ref{2}), the set $J:=\{n\in \omega
:X_{n}\cap O_{\alpha _{0}}=\varnothing \}$ is finite. For each $n\in J,$
there is $x_{n}\in D\cap X_{n}$, and there is $\alpha _{n}\in I$ such that $%
x_{n}\in O_{\alpha _{n}}.$ Then $\bigcup_{n\in J\cup \{0\}}O_{\alpha
_{n}}\in \mathcal{D}$.

On the other hand, if $\mathcal{C}$\ $\subseteq \mathcal{O}(\infty )$\ is a
compact family, there is $n_{0}$\ such that $\mathcal{C}$ is free on $%
X_{n_{0}}$, that is, the restriction of $\mathcal{C}$ to $X_{n_{0}}$ is
finer than the cofinite filter of $X_{n_{0}}$. Otherwise, for each $n\in
\omega $ there would be a finite subset $F(n)$ of $X_{n}$ such that $F(n)\#%
\mathcal{C}.$

On the other hand, $V:=X\setminus \bigcup_{n\in \omega }F(n)\in \mathcal{O}%
(\infty )$, hence, by the compactness of $\mathcal{C}$, there exists a
finite set $T$ such that $V\cup T\in \mathcal{C}$ and $\left( V\cup T\right)
\cap \bigcup_{n\in \omega }F(n)\subset T$. This is a contradiction with $%
F(n)\#\mathcal{C}$ for all $n.$

By the compactness of $\mathcal{C}$ there exists $C_{0}\in \mathcal{C}$ such
that $C_{0}\cap X_{n_{0}}$ is finite. As $\mathcal{C}$ is free on $X_{n_{0}}$%
, there is $C\in \mathcal{C}$ disjoint from $C_{0}\cap X_{n_{0}}$, so that $%
\mathcal{C\vee C}\nsubseteq \mathcal{D}$. Hence $X$ is not infraconsonant.
\end{proof}


\begin{example}
\label{ex:Arens}\emph{[The Arens space is not infraconsonant]} Recall that
the Arens space has underlying set $\{\infty \}\cup \{x_{n,k}:n\in \omega
,k\in \omega \}$ and carries the topology in which every point but $\infty $
is isolated and a base of neighborhoods of $\infty $ is given by sets of the
form%
\begin{equation*}
\{\infty \}\cup \bigcup_{n\geq p}\{x_{n,k}:k\geq f(n)\},
\end{equation*}%
where $p$ ranges over $\omega $ and $f$ ranges over $\omega ^{\omega }.$
\end{example}

In \cite{DSW} a notion of \emph{sequential contour} of arbitrary order was
introduced. A sequential contour of rank $1$ is a free sequential filter
(that is, the cofinite filter of a countable set); a sequential contour of
rank $\alpha >1$ is a contour of the sequence $(\mathcal{F}_{n})_{n<\omega }$%
, where $\mathcal{F}_{n}$ is a sequential contour of rank $\alpha _{n}$ on a
countable set $X_{n}$, $\left\{ X_{n}:n<\omega \right\} $ are disjoint, and $%
\alpha =\sup_{n<\omega }(\alpha _{n}+1)$. It follows that

\begin{corollary}
For every countable ordinal $\alpha $ the prime topology determined by a
sequential contour of rank $\alpha $ is not infraconsonant.
\end{corollary}

We do not know however if there are completely regular infraconsonant spaces
that are not consonant.

A. Bouziad pointed out to us that, in view of Proposition \ref{prop:join},
Theorem \ref{th:comtour} also shows that the assumption of separation is
essential in the result of J. Lawson stating that a compact Hausdorff
semitopological lattice is topological \cite{lawson.semilattice}. Indeed, if 
$X$ is regular but not infraconsonant (for instance the Arens space), then $%
C(X,\$^{\ast })$ is a $T_{0}$ compact semitopological lattice (i.e., $\cap $
is separately continuous) which is not a topological lattice, because $\cap $
is not jointly continuous.

\section{Continuity of translations}

As we mentioned, Francis Jordan introduced in \cite{francis.coincide} the 
\emph{fine Isbell topology }on the set $C(X,Y)$. We shall now prove that
translations are always continuous for the fine Isbell topology, and that
the neighborhood filters at the zero function $\overline{0}$ for the Isbell
and for the fine Isbell topologies coincide. Therefore translations are
continuous for the Isbell topology if and only if it coincides with the fine
Isbell topology.

If $N$ and $M$ are two subsets of $X\times Y$, the set $N$ is \emph{buried
in }$M,$ in symbols $N\ll M,$ if for every $x\in X$ there exists $V\in 
\mathcal{O}_{X}(x)$ and $W\in \mathcal{O}_{Y}(N(x))$ such that $V\times
W\subseteq M$. If $f\in C(X,Y)$ and $A\subseteq X,$ we denote by $f_{|A}$
the graph of the restriction of $f$ to $A$. A subbase for the fine Isbell
topologies is given by sets of the form:%
\begin{equation*}
\left\langle \mathcal{A},M\right\rangle :=\left\{ f\in C(X,Y):\exists A\in 
\mathcal{A},f_{|A}\ll M\right\} ,
\end{equation*}%
where $\mathcal{A}$ ranges over compact families of $X$ and $M$ ranges over
open subsets of $X\times Y$. We denote by $C_{\overline{\kappa }}(X,Y)$ the
set $C(X,Y)$ endowed with the fine Isbell topology. If $(G,+)$ is a
topological group, we denote by $0$ its neutral element and by $\overline{0}$
the constant function zero of $C(X,G).$

\begin{theorem}
\label{th:coincideat0} Let $(G,+)$ be a topological group. The fine Isbell
topological space $C_{\overline{\kappa }}(X,G)$ is invariant by translations.

The neighborhood filters at $\overline{0}$ for the fine Isbell and the
Isbell topologies coincide.
\end{theorem}

\begin{proof}
(1) $\mathcal{N}_{\kappa }(\overline{0})\leq \mathcal{N}_{\overline{\kappa }%
}(\overline{0}):$ is clear. Consider now $\left\langle \mathcal{A}%
,M\right\rangle $ such that $\overline{0}\in \left\langle \mathcal{A}%
,M\right\rangle ,$ $\mathcal{A}\in \kappa (X)$ and $M$ is open in $X\times G$%
. There is $A\in \mathcal{A}$ such that for every $x\in A,$ there is $%
V_{x}\in \mathcal{O}(x)$ and $W_{x}\in \mathcal{O}_{G}(0)$ such that $%
V_{x}\times W_{x}\subseteq M.$ Since $\mathcal{A}$ is compact and $%
A=\bigcup\limits_{x\in A}V_{x}$ there is a finite subset $F$ of $A$ such
that $B=\bigcup\limits_{x\in F}V_{x}\in \mathcal{A}.$ But then $%
W=\bigcap\limits_{x\in F}W_{x}\in \mathcal{O}_{G}(0)$ and $B\times
W\subseteq M$ so that 
\begin{equation*}
\overline{0}\in \left[ \mathcal{A}\downarrow B,W\right] \subseteq
\left\langle \mathcal{A},M\right\rangle .
\end{equation*}

(2) $\mathcal{N}_{\overline{\kappa }}(f)\geq f+\mathcal{N}_{\kappa }(%
\overline{0}):$ Let $\mathcal{A}\in \kappa (X)$, $B\in \mathcal{O}_{G}(0)$.
Consider $M:=\bigcup\limits_{x\in X}\{x\}\times \left( f(x)+B\right) $. Then 
$f\in \left\langle \mathcal{A},M\right\rangle $ and $\left\langle \mathcal{A}%
,M\right\rangle \subseteq f+[\mathcal{A},B]$. Indeed, if $h\in \left\langle 
\mathcal{A},M\right\rangle $ then there is $A\in \mathcal{A}$ such that for
all $x\in A$, there is an open neighborhood $V_{x}$ of $x$ and an open
neighborhood $W_{x}$ of $h(x)$ such that $V_{x}\times W_{x}\subseteq M$. In
particular, $\{x\}\times W_{x}\subseteq M$ so that $W_{x}\subseteq f(x)+B$
and $\left( h-f\right) (x)\in B$. Therefore $\left( h-f\right) \left(
A\right) \subseteq B$.

(3) $\mathcal{N}_{\overline{\kappa }}(f)\leq f+\mathcal{N}_{\kappa }(%
\overline{0}):$ Let $\mathcal{A}\in \kappa (X)$ and let $M$ be an open
subset of $X\times G$ such that $f\in \left\langle \mathcal{A}%
,M\right\rangle ,$ that is, there is $A\in \mathcal{A}$ such that for all $%
x\in A$, there is an open neighborhood $V_{x}$ of $x$ and an open
neighborhood $W_{x}=f(x)+B_{x}$ of $f(x),$ where $B_{x}\in \mathcal{O}%
_{G}(0) $ such that $V_{x}\times W_{x}\subseteq M$. By continuity of $f$ we
may assume that $f(V_{x})\subseteq f(x)+B_{x}^{\prime }\subseteq W_{x}$ for
each $x,$ where $B_{x}^{\prime }\in \mathcal{O}_{G}(0)$ and $B_{x}^{\prime
}+B_{x}^{\prime }\subseteq B_{x}$. Since $\mathcal{A}$ is compact and $%
A=\bigcup\limits_{x\in A}V_{x}$ there is a finite subset $F$ of $A$ such
that $A_{1}=\bigcup\limits_{x\in F}V_{x}\in \mathcal{A}.$ Let $W\in \mathcal{%
O}_{G}(0)$ be such that $W=-W$ and $W\subseteq \bigcap\limits_{x\in
F}B_{x}^{\prime }\in \mathcal{O}_{G}(0)$. Then $f+\left[ \mathcal{A}%
\downarrow A_{1},W\right] \subseteq \left\langle \mathcal{A},M\right\rangle $%
. Indeed, if $h\in \left[ \mathcal{A}\downarrow A_{1},W\right] $ then there
is $A_{2}\in \mathcal{A},$ $A_{2}\subseteq A_{1}$ such that $%
h(A_{2})\subseteq W$. For each $x\in A_{2},$ there is $t_{x}\in F$ such that 
$x\in V_{t_{x}}$. Note that $V_{t_{x}}\times W_{t_{x}}\subseteq M$ and that $%
f(x)\in f(V_{t_{x}})$ and 
\begin{equation*}
f(V_{t_{x}})+W\subseteq f(x)+B_{x}^{\prime }+B_{x}^{\prime }\subseteq
W_{t_{x}}
\end{equation*}%
so that $V_{t_{x}}\times \left( f(x)+W\right) \subseteq M$ which completes
the proof because $f(x)+W\in \mathcal{O}\left( \left( f+h\right) (x)\right) $
and $V_{t_{x}}\in \mathcal{O}(x)$.
\end{proof}

\begin{corollary}
\label{cor:invariance} $C_{\kappa }(X,G)$ is invariant by translation if and
only if $C_{\kappa }(X,G)=C_{\overline{\kappa }}(X,G).$
\end{corollary}

The result above also provides a more handy description of the fine Isbell
topology on $C(X,G)$ when $G$ is a topological group (for instance for $C_{%
\overline{\kappa }}(X,\mathbb{R})$) :

\begin{equation*}
\mathcal{N}_{\overline{\kappa }}(f)=f+\left\{ \left[ \mathcal{A},B\right] :%
\mathcal{A}\in \kappa (X),B\in \mathcal{O}_{G}(0)\right\} .
\end{equation*}

\begin{theorem}
\label{th:finegroup}The following are equivalent:

\begin{enumerate}
\item $C_{\overline{\kappa }}(X,\mathbb{R})$ is a topological vector space;

\item $C_{\overline{\kappa }}(X,\mathbb{R})$ is a topological group;

\item $X$ is infraconsonant.
\end{enumerate}
\end{theorem}

\begin{proof}
If $C_{\overline{\kappa }}(X,\mathbb{R})$ is a topological group then $%
\mathcal{N}_{\overline{\kappa }}(\overline{0})+\mathcal{N}_{\overline{\kappa 
}}(\overline{0})=\mathcal{N}_{\overline{\kappa }}(\overline{0}).$ But $%
\mathcal{N}_{\overline{\kappa }}(\overline{0})=\mathcal{N}_{\kappa }(%
\overline{0})$ so that by Theorem \ref{th:at0}, $X$ is infraconsonant.
Conversely, if $X$ is infraconsonant, then $\mathcal{N}_{\overline{\kappa }}(%
\overline{0})+\mathcal{N}_{\overline{\kappa }}(\overline{0})=\mathcal{N}_{%
\overline{\kappa }}(\overline{0})$ and translations are continuous in $C_{%
\overline{\kappa }}(X,\mathbb{R})$ so that $C_{\overline{\kappa }}(X,\mathbb{%
R})$ is a topological group. Remains to see that if $C_{\overline{\kappa }%
}(X,\mathbb{R})$ is a topological group, it is also a topological vector
space.

First, note that for each fixed $f\in C(X,\mathbb{R})$ the map $S_{f}:%
\mathbb{R}\rightarrow C_{\overline{\kappa }}(X,\mathbb{R})$ defined by $%
S_{f}(r)=rf$ is continuous. Indeed, for each $\mathcal{A}\in \kappa (X),$
there is $A\in \mathcal{A}$ and $R\in \mathbb{R}$ such that $f(A)\subseteq
B(0,R)$ by Proposition \ref{prop:bounded}. Therefore for each $O\in \mathcal{%
N}_{\mathbb{R}}(0)$ there is $\delta >0$ such that $B(0,\delta )\cdot
f(A)\subseteq O.$ Thus, if $rf+[\mathcal{A},O]\in \mathcal{N}_{\overline{%
\kappa }}(f)$ then $B(r,\delta )\cdot f=rf+B(0,\delta )\cdot f\subseteq rf+[%
\mathcal{A},O].$

Note also that $S:\mathbb{R\times }C_{\overline{\kappa }}(X,\mathbb{R}%
)\rightarrow C_{\overline{\kappa }}(X,\mathbb{R})$ defined $S(r,f)=rf$ is
continuous at $(r,\overline{0})$ for each $r$ by Proposition \ref%
{prop:scalarmultIsbell}. Let now $r\in \mathbb{R}$ and $f\in C(X,\mathbb{R})$
and consider $rf+W\in $ $\mathcal{N}_{\overline{\kappa }}(f),$ where $W\in 
\mathcal{N}_{\overline{\kappa }}(\overline{0}).$ Since $C_{\overline{\kappa }%
}(X,\mathbb{R})$ is a topological group, there is $V\in $ $\mathcal{N}_{%
\overline{\kappa }}(\overline{0})$ such that $V+V\subseteq W.$ By continuity
of $S_{f},$ there is $T\in \mathcal{N}_{\mathbb{R}}(r)$ such that $T\cdot
f\subseteq rf+V$. Moreover, by continuity of $S$ at $(r,\overline{0})$ there
is $T^{\prime }\in \mathcal{N}_{\mathbb{R}}(r),$ $T^{\prime }\subseteq T$
and $U\in \mathcal{N}_{\overline{\kappa }}(\overline{0})$ such that $%
T^{\prime }\cdot U\subseteq V$. Then 
\begin{equation*}
T^{\prime }\cdot \left( f+U\right) =T^{\prime }\cdot f+T^{\prime }\cdot
U\subseteq rf+V+V\subseteq rf+W,
\end{equation*}%
which proves continuity of $S$ at $(r,f)$ because $f+U\in \mathcal{N}_{%
\overline{\kappa }}(f).$
\end{proof}

In particular, in view of Example \ref{ex:Arens}, the fine Isbell topology
does not need to be a group topology. In \cite[Example 1]{francis.coincide}
Jordan shows that if $X$ and $Y$ are two zero-dimensional consonant spaces
such that the topological sum $Z:=X\oplus Y$ is not consonant then $%
C_{\kappa }(Z,\mathbb{R})<C_{\overline{\kappa }}(Z,\mathbb{R})$. In view of
Corollary \ref{cor:invariance} we obtain:

\begin{example}
\label{ex:nonsplittable} There exists a space $Z$ such that translations of $%
C_{\kappa}(Z,\mathbb{R})$ fail to be continuous.
\end{example}

Following \cite{francis.coincide}, we say that a space $X$ is \emph{%
sequentially inaccessible }provided that for any sequence $\left( \mathcal{F}%
_{n}\right) _{n\in \omega }$ of countably based $z$-filters (filters based
in zero sets) on $X$%
\begin{equation*}
\left( \forall _{n}\;\adh\mathcal{F}_{n}=\varnothing \right)
\Longrightarrow \adh(\bigcap_{n\in \omega }\mathcal{F}%
_{n})=\varnothing .
\end{equation*}%
Jordan proved \cite{francis.coincide} that if $X$ is completely regular,
Lindel\"{o}f and sequentially inaccessible, then the fine Isbell topology
and the natural topology coincide on $C(X,\mathbb{R}).$ This result can be
combined with Theorem \ref{th:finegroup} to obtain the following
Banach-Dieudonn\'{e} like result \footnote{%
The classical theorem of Banach-Dieudonn\'{e} (and its many variants) can be
seen as providing sufficient conditions on a topological vector space for
the natural topology on its dual space to be a group topology. See \cite%
{jarchowbook}, \cite{BDrevisited} for details.}.

\begin{theorem}
If $X$ is a completely regular, Lindel\"{o}f and sequentially inaccessible
space, then the natural topology on $C(X,\mathbb{R})$ is a group (and
vector) topology if and only if $X$ is infraconsonant.
\end{theorem}

As we have seen, translations in $C(X,\mathbb{R})$ are, in general, not
continuous for the Isbell topology. They are however continuous if $X$ is
prime (that is, has at most one non-isolated point). More generally,

\begin{proposition}
\label{prop:prime+trans} If $X$ is prime and $G$ is an abelian consonant
topological group, then the Isbell topology on $C(X,G)$ is translation
invariant.
\end{proposition}

\begin{proof}
If a family $\mathcal{A}$ is compact on $X$, then for each $x\in X$ the
family $\mathcal{A}_{x}:=\mathcal{O}_{X}(x)\cap \mathcal{A}$ is compact
included in $\mathcal{A}$. Therefore it is enough to consider basic
neighborhoods for the Isbell topology of the form $[\mathcal{A}_{x},U]$
where $U$ is an open subset of $G$.

\qquad Consider $f\mapsto g+f$ for $g\in C(X,G)$. Let $x\in X$ and $\mathcal{%
A}$ be a compact family included in $\mathcal{O}_{X}(x)$ and $U$ is an open
subset of $G$. Let $f_{0}+g\in \lbrack \mathcal{A},U]$.

\qquad The family $\mathcal{D}:=\mathcal{O}_{G}(\left( f_{0}+g\right) (%
\mathcal{A}))$ is compact in a consonant space $G$, hence there is a compact
set $K\subset U$ such that $\mathcal{O}_{G}(K)\subset \mathcal{D}$. Let $%
W=-W $ be a closed neighborhood of $0$ in $G$ such that $K+3W\subset U$.
Then there is $A\in \mathcal{A}$ such that $\left( f_{0}+g\right) (A)\subset
K+W$. Furthermore there exists $A_{0}\in \mathcal{A}$ such that $%
A_{0}\subset A$ and $f_{0}(A_{0})$ is bounded and $\left( f_{0}+g\right)
(A_{0})\subset K+W$.

\qquad Let $V_{0}$ be an element of $\mathcal{O}_{X}(x)$ included in $A_{0}$
such that $f_{0}(V_{0})\subset f_{0}(x_{\infty })+W$ and $g(V_{0})\subset
g(x_{\infty })+W$.

\qquad Then there is a finite subset $F$ of $A_{0}$ (disjoint from $V_{0}$)
such that $A_{1}:=V_{0}\cup F\in \mathcal{A}$. Then let $\mathcal{A}_{1}:=%
\mathcal{A}\downarrow A_{1}$. Of course, $f_{0}+g\in \lbrack \mathcal{A}%
_{1},K+W]$ and $\mathcal{A}_{1}\in \kappa (X)$.

\qquad Then there is $n<\omega $ and finite sets $F_{1},\ldots ,F_{n}$ such
that $f_{0}(F_{k})-$ $f_{0}(F_{k})\subset W$ and $g(F_{k})-g(F_{k})\subset W$
for each $1\leq k\leq n$, and moreover $F_{1}\cup \ldots \cup F_{n}=F$.
Finally, let $\mathcal{D}_{0}:=\mathcal{A}_{1}\vee V_{0}$ and $\mathcal{D}%
_{k}:=\mathcal{A}_{1}\vee F_{k}$ for $1\leq k\leq n$. Then $\mathcal{A}%
_{1}=\bigcap\limits_{k=0}^{n}\mathcal{D}_{k}$. On the other hand, there
exist $x_{k}\in F_{k}$ for $1\leq k\leq n$, such that%
\begin{equation*}
f_{0}\in \bigcap\limits_{k=0}^{n}[\mathcal{D}_{k},f_{0}(x_{k})+W],
\end{equation*}%
where $x_{0}:=x_{\infty }$. If now $f\in \bigcap_{k=0}^{n}[\mathcal{D}%
_{k},f_{0}(x_{k})+2W]$ then%
\begin{equation*}
f+g\in \bigcap_{k=0}^{n}[\mathcal{D}_{k},f_{0}(x_{k})+g(x_{k})+3W]\subset
\lbrack \mathcal{A}_{1},U]\subset \lbrack \mathcal{A},U].
\end{equation*}
\end{proof}

\begin{corollary}
\label{thm:first} If $X$ is a prime topological space, then $C_{\kappa }(X,%
\mathbb{R})$ is translation invariant.
\end{corollary}

Since, by Corollary \ref{cor:invariance}, $C_{\kappa }(X,\mathbb{R})$ is
translation invariant if and only if it coincides with $C_{\overline{\kappa }%
}(X,\mathbb{R})$, Corollary \ref{thm:first} implies a result \cite[Theorem 18%
]{francis.coincide} of Jordan that the Isbell and the fine Isbell topologies
coincide provided that the underlying topology is prime.

Corollary \ref{thm:first} combined with Theorem \ref{th:at0} and Corollary %
\ref{cor:invariance} leads to the following results.

\begin{theorem}
\label{cor} If $X$ is prime, then $C_{\kappa }(X,\mathbb{R})$ is a
topological group (or topological vector space) if and only if $X$ is
infraconsonant.
\end{theorem}

Note that if $X$ is as in Example \ref{ex:Arens}, then $C_{\kappa }(X,%
\mathbb{R})$ is invariant by translation but not a topological group. 


\end{document}